\documentclass[12pt]{article}
\usepackage{stmaryrd}
\usepackage{bbm}
\usepackage[french,english]{babel}

\usepackage{setspace,mathtools,amsmath,amssymb,amsthm,fullpage,outline,centernot,slashed} 
\usepackage[pdftex]{graphicx}
\usepackage{mathrsfs}
\usepackage{enumerate}
\usepackage{accents}
\usepackage[colorlinks=true,citecolor=black,linkcolor=black,urlcolor=blue]{hyperref}
\usepackage{upgreek}

\theoremstyle{plain}
\newtheorem{theorem}{Theorem}
\newtheorem{lemma}[theorem]{Lemma}

\theoremstyle{definition}

\theoremstyle{remark}
\newtheorem{remark}[theorem]{Remark}

\newcommand{\Z}{\mathbb{Z}}

\newcommand{\R}{\mathbb{R}}

\newcommand{\half}{\frac{1}{2}}
\newcommand{\wo}{\backslash}

\newcommand{\ben}{\begin{enumerate}}
\newcommand{\een}{\end{enumerate}}
\newcommand{\bit}{\begin{itemize}}
\newcommand{\eit}{\end{itemize}}
\def\bal#1\eal{\begin{align*}#1\end{align*}}

\DeclareMathSymbol{\mlq}{\mathord}{operators}{``}
\DeclareMathSymbol{\mrq}{\mathord}{operators}{`'}

\def\XXint#1#2#3{{\setbox0=\hbox{$#1{#2#3}{\int}$ }
\vcenter{\hbox{$#2#3$ }}\kern-.6\wd0}}

\title{Centered Hardy--Littlewood maximal operator on the real line: lower bounds}
\author{Paata Ivanisvili, Princeton University
\and Samuel Zbarsky, Princeton University}

\begin{document}
\maketitle
\begin{abstract}
For $1<p<\infty$ and $M$ the centered Hardy-Littlewood maximal operator on $\R$, we consider whether there is some $\varepsilon=\varepsilon(p)>0$  such that $||Mf||_p\ge (1+\varepsilon)||f||_p$. We prove this for $1<p<2$. For $2\le p<\infty$, we prove the  inequality for indicator functions and for unimodal functions.
\end{abstract}

\selectlanguage{french} 
\begin{abstract}
Soient $1<p<\infty$ et $M$ la fonction maximale de Hardy-Littlewood sur $\R$. Nous \'etudions l'existence d'un $\varepsilon=\varepsilon(p)>0$ tel que $||Mf||_p\ge (1+\varepsilon)||f||_p$. Nous l'\'etablissons pour $1<p<2$. Pour $2\le p<\infty$, nous prouvons l'in\'egalit\'e pour les fonctions indicatrices  et les fonctions unimodales.
\end{abstract}
\selectlanguage{english} 
\section{Introduction}
Given a locally integrable real-valued function  $f$ on $\mathbb{R}^{n}$ define its uncentered maximal function  $M_{u}f(x)$  as follows 
\begin{align}\label{maxu}
M_{u}f(x) = \sup_{B \ni x} \frac{1}{|B|} \int_{B}|f(y)|dy,
\end{align}
where the supremum is taken over all balls $B$ in $\mathbb{R}^{n}$ containing the point $x$, and  $|B|$ denotes the Lebesgue volume of $B$. 
In studying {\em lower operator norms} of the maximal function \cite{Lerner2010}  A.~Lerner raised the following question:  given $1<p<\infty$  can one find a constant  $\varepsilon=\varepsilon(p)>0$ such that 
\begin{align}\label{lerq}
\|M_{u}f\|_{L^{p}(\mathbb{R}^{n})} \geq (1+\varepsilon) \|f\|_{L^{p}(\mathbb{R}^{n})} \quad \text{for all} \quad f \in L^{p}(\mathbb{R}^{n}).
\end{align}
The affirmative answer was obtained in \cite{IvanisviliJayeNazarov2017}, i.e., the Lerner's inequality (\ref{lerq}) holds for all $1<p<\infty$ and for any $n\geq 1$. The paper also studied the estimate (\ref{lerq}) for other maximal functions. For example, the lower bound  (\ref{lerq}) persists if one takes supremum in (\ref{maxu}) over the shifts and dilates of a fixed  centrally symmetric convex body $K$. Similar positive results have been obtained for dyadic maximal functions \cite{Melas2017}; maximal functions defined over {\em $\lambda$-dense family} of sets, and {\em almost centered} maximal functions (see \cite{IvanisviliJayeNazarov2017} for details). 

The Lerner's inequality for the centered maximal function 
\begin{align}\label{lera}
\|Mf\|_{L^{p}(\mathbb{R}^{n})} \geq (1+\varepsilon(p,n))\|f\|_{L^{p}(\mathbb{R}^{n})}, \quad f \in L^{p}(\mathbb{R}^{n}), \qquad Mf(x) = \sup_{r>0}\frac{1}{|B_{r}(x)|} \int_{B_{r}(x)} |f|,
\end{align}
where the supremum is taken over all balls centered at $x$,  is an open question, and  the full characterization of the pairs  $(p,n)$, $n\geq 1$, and $ 1<p<\infty$, for which (\ref{lera}) holds with some $\varepsilon(p,n)>0$ and for all $f \in L^{p}(\mathbb{R}^{n})$ seems to be unknown.  If $n \geq 3$, and  $p>\frac{n}{n-2}$ then  one can show that $f(x) = \min\{|x|^{n-2}, 1\} \in L^{p}(\mathbb{R}^{n})$, and $Mf(x)=f(x)$, as $f$ is the pointwise minimum of two superharmonic functions. This gives a counterexample to (\ref{lera}). In fact, Korry \cite{Korry2001} proved that the centered maximal operator does not have fixed points unless $n \geq 3$ and  $p>\frac{n}{n-2}$, but a lack of fixed points does not imply that \eqref{lera} holds. On the other hand for any $n\geq 1$,  by comparing $Mf(x) \geq C(n) M_{u}f(x)$, and using the fact that $\| M_{u}f\|_{L^{p}(\mathbb{R}^{n})} \geq (1+\frac{B(n)}{p-1})^{1/p}\|f\|_{L^{p}(\mathbb{R}^{n})}$ (see \cite{IvanisviliJayeNazarov2017}), one can easily conclude that (\ref{lera}) holds true whenever $p$ is sufficiently close to $1$. It is natural to ask what is the maximal $p_{0}(n)$ for which  if $1<p<p_{0}(n)$ then (\ref{lera}) holds. 
\subsection{New results}
In this paper we study the case of dimension $n=1$ and the centered Hardy--Littlewood maximal operator $M$. We obtain
\begin{theorem}\label{cth1}
If $1<p<2$  and $n=1$ then Lerner's inequality (\ref{lera}) holds true, namely
\[
\|Mf\|_{p}\geq \left(\frac{p}{2(p-1)}\right)^{1/p}\|f\|_{p}.
\]
\end{theorem}
\begin{theorem}\label{cth2}
For $n=1$, and  any $p, 1<p<\infty$, inequality (\ref{lera}) holds true a) for the class of indicator functions with $\epsilon(p,n)=1/4^p$, and b) for the class of unimodal functions, with $\epsilon(p,n)$ not explicitly given.
\end{theorem}
\section{Proof of the main results}
\subsection{Proof of Theorem~\ref{cth1}}\label{sec:smallp}
First we prove the following modification of the classical Riesz's  sunrise lemma (see Lemma~1 in \cite{Grecia}). Our proof is similar to the proof of the lemma.
\begin{lemma}\label{lem:pseudosunrise}
For a nonnegative continuous compactly supported  $f$ and any $\lambda>0$, we have
\[
|\{Mf\ge\lambda\}|\ge \frac{1}{2\lambda}\int_{\{f\ge\lambda\}} f.
\]
\end{lemma}
\begin{proof}
Define an auxiliary function $\varphi(x)$ via
\[
\varphi(x)=\sup_{y<x}\int_y^x f(t)dt-2\lambda(x-y).
\]

Notice that if $f(x)>2\lambda$ then  $\varphi(x)>0$. Indeed, 
\begin{align}\label{adef}
\varphi(x) = \sup_{y<x }\;  (x-y) \left[ \frac{1}{x-y}\int_{y}^{x}f  - 2\lambda \right] >0,
\end{align}
because we can choose $y$ sufficiently close to $x$, and use the fact that $\lim_{y \to x} \frac{1}{x-y}\int_{y}^{x}f = f(x)$. 
On the other hand if $\varphi(x)>0$, then $Mf(x)>\lambda$. Indeed,  it follows from (\ref{adef}) that $\sup_{y<x} \frac{1}{x-y}\int_{y}^{x} f >2\lambda$. Therefore 
\begin{align*}
Mf(x) = \sup_{r>0} \frac{1}{2r} \int_{x-r}^{x+r}f  \geq \frac{1}{2(x-y)}\int_{y}^{x} f \geq \lambda.
\end{align*}

Thus, we obtain 
\begin{align}
&\{Mf\ge\lambda\}\supseteq\{f\ge\lambda\}\cup\{\varphi>0\}; \label{supset}\\
&\{f>2\lambda\}\subseteq\{\varphi>0\}. \label{subset}
\end{align}
Therefore, it follows that 
\bal
|\{Mf\ge\lambda\}|&\ge |\{\varphi>0\}|+|\{\lambda\le f\le 2\lambda\}\wo\{\varphi>0\}|\\
&\ge \frac{1}{2\lambda}\int_{\{\varphi>0\}} f+\frac{1}{2\lambda}\int_{\{\lambda\le f\le 2\lambda\}\wo\{\varphi>0\}} f\\
&\ge \frac{1}{2\lambda}\int_{\{f\ge\lambda\}} f.
\eal
\end{proof}
Now we are ready to  prove Theorem~\ref{cth1}. 
\begin{proof}[Proof of Theorem~\ref{cth1}] Take any continuous bounded compactly supported $f\geq 0$. By Lemma~\ref{lem:pseudosunrise}, for any $\lambda>0$  we have 
\begin{align}\label{lb1}
|\{Mf\geq \lambda\}| \geq \frac{1}{2\lambda} \int_{\mathbb{R}} f(x) \mathbbm{1}_{[\lambda, \infty)}(f(x)) dx.
\end{align}
Finally we multiply both sides of (\ref{lb1}) by $p\lambda^{p-1}$, and we integrate the obtained inequality in $\lambda$ on $(0, \infty)$, so we obtain 
\begin{align*}
\int_{\mathbb{R}} (Mf)^{p} \geq \int_{0}^{\infty}\int_{\mathbb{R}} \frac{p \lambda^{p-2}}{2} f(x) \mathbbm{1}_{[\lambda, \infty)}(f(x)) dx d\lambda  = \frac{p}{2(p-1)}\int_{\mathbb{R}} f^{p},
\end{align*}
and $\frac{p}{2p-2}>1$ precisely when $p<2$. This finishes the proof of Theorem~\ref{cth1} for continuous compactly supported bounded nonnegative $f$. To obtain the inequality $\|Mf\|_{p}\geq (\frac{p}{2(p-1)})^{1/p}\|f\|_{p}$ for an arbitrary nonnegative $f \in L^{p}(\mathbb{R})$ we can approximate $f$ in $L^p$ by a sequence of compactly supported smooth functions $f_{n}$, and use the fact that the operator $M$ is Lipschitz on $L^p$ (since it is bounded and subadditive).
\end{proof}
\begin{remark}
The argument presented above is a certain modification of the classical Riesz's  sunrise lemma, and an adaptation of an argument of Lerner (see Section 4 in ~\cite{Lerner2010}).  For $p$ less than about 1.53, it is possible to use Lerner's result directly, together with the fact that $Mf\ge (M_uf)/2$. We need the modified sunrise lemma to get the result for all $p<2$.
\end{remark}

\subsection{Proof of Theorem~\ref{cth2}}\label{sec:indicator} 
\subsubsection{Indicator functions}
\begin{proof}[Proof of Theorem~\ref{cth2} for indicator functions $\mathbbm{1}_{E}$] Let $\mathbbm{1}_{E} \in L^{p}(\mathbb{R})$ and let $\hat\delta>0$. We approximate $\mathbbm{1}_{E}$ arbitrarily well in $L^p$ by a nonnegative continuous compactly supported function $f$. Then $f$ approximates $\mathbbm{1}_{E}$ and $Mf$ also approximates $M\mathbbm{1}_E$ to within some $\delta\ll\hat\delta$ in $L^p$.

For a.\ e. $x\in E$, we have $M\mathbbm{1}_E(x)\ge 1$. Additionally, by Lemma~\ref{lem:pseudosunrise}, we have that
\[
|\{Mf\ge 1/4\}|\ge 2\int_{\{f\ge\frac{1}{4}\}} f\ge 2\int_{\{f\ge\frac{1}{4}\}\cap E}\mathbbm{1}_{E} -2\int_{E} |\mathbbm{1}_{E}-f|.
\]
By making $\delta$ is small, we can ensure that $\{|f-\mathbbm{1}_{E}|\ge 3/4\}$ is small, so
\[
2\int_{\{f\ge\frac{1}{4}\}\cap E}\mathbbm{1}_{E}\ge 2|E|-\hat\delta/2.
\]
Also, by Holder's inequality, we can bound $\int_{E} |\mathbbm{1}_{E}-f|$ in terms of $||\mathbbm{1}_{E}-f||_p<\delta$. Thus, when $\delta$ is sufficiently small, we get
\[
|\{Mf\ge 1/4\}|\ge  2|E|-\hat\delta,
\]
so there is a set of measure at least $|E|-\hat\delta$ on which $\mathbbm{1}_E=0$ and  $Mf \ge 1/4$. If $\delta$ is sufficiently small, we have that $|\{Mf-M\mathbbm{1}_E\ge \hat\delta\}|<\hat\delta$, so there is a set of measure $|E|-2\hat\delta$ on which $\mathbbm{1}_E=0$ and  $M\mathbbm{1}_E \ge 1/4-\hat\delta$. Taking $\hat\delta\to 0$, we get
\[
\|M\mathbbm{1}_{E}\|_p^p \ge (1+1/4^p) \|\mathbbm{1}_{E}\|_p^p.
\]
\end{proof}
\subsubsection{Unimodal functions}\label{sec:unimodal}

Next we obtain lower bounds on $L^{p}$ norms of the  maximal operator over the class of unimodal functions. By unimodal function $f \in L^{p}(\mathbb{R})$, $f\geq 0$,  we mean any function which is  increasing until some point $x_{0}$ and then decreasing. Without loss of generality we will assume that $x_{0}=0$. 
\begin{proof}[Proof of Theorem~\ref{cth2} for unimodal functions]
We can assume that  $\|f\mathbbm{1}_{\R^+}\|_p^p\ge\half\|f\|_p^p$.

Let $\tilde f=f\mathbbm{1}_{\R^+}$. We define $M^{n} = \underbrace{M\circ \cdots \circ M}_{n}$ to be the $n$-th iterate of $M$. We will find an $n$, independent of $f$, such that $\|M^n \tilde f\|_p^p>2^{p+1}\|\tilde f\|_p^p$, independent of the function $f$.
First, for $x>0$, let 
\[
a(x)=\min_{k\in\Z,2^k>x}  2^k.
\]
Then let
\[
\psi(x)=\tilde f(a(x)),
\]
that is $\psi\le \tilde f$, and $\psi$ is a step function approximation from below. Then

\[
2 \|\psi\|_p^p = 2 \sum_{k\in \Z} 2^k \tilde f(2^{k+1})^p = \sum_{s\in \Z} 2^s \tilde f(2^s)^p \geq \|\tilde{f}\|_p^p
\]
Now let
\[
\bar g(x)=(1-\sqrt{x})\mathbbm{1}_{(0,1]}(x).
\]
Then for $0< x\le 9/8$, we have that 
\[
M\bar g(x)\ge\frac{1}{2x}\int_0^{2x}\bar g(y)dy\ge\frac{1}{2x}\int_0^{2x}1-\sqrt{y}dy=1-\frac{2}{3}\sqrt{2x}=\bar g(8x/9),
\]
and for all $x\notin (0,9/8]$, we have $M\bar g(x)\ge 0=\bar g(8x/9)$. Thus
\[
M^n\bar g(x)\ge \bar g\left((8/9)^nx\right),
\]
so
\begin{equation}\label{eq:Mn1}
\int_{\half(9/8)^n}^{(9/8)^n} (M^n\mathbbm{1}_{(0,1]})^p\ge \int_{\half(9/8)^n}^{(9/8)^n} (M^ng)^p\ge (9/8)^n \int_{\half}^{1} \bar g^p=C_p(9/8)^n.
\end{equation}
Note that for all $k\in\Z$, we have $\psi\ge \tilde f(2^{k+1}) \mathbbm{1}_{(2^k,2^{k+1}]}$. Thus
\[
M^n\psi(x)\ge  \tilde f(2^{k+1}) M^n\mathbbm{1}_{(2^k,2^{k+1}]}(x).
\]
We will use this lower bound for varying values of $k$ for different $x$. We use \eqref{eq:Mn1} in the third inequality below, since $\mathbbm{1}_{(2^k,2^{k+1}]}$ is just a horizontal rescaling and translation of $\mathbbm{1}_{(0,1]}$. We have
\bal
\|M^n\psi\|_p^p&\ge \sum_{-\infty}^\infty \int_{2^k+(9/8)^n2^{k-1}}^{2^k+(9/8)^n2^k} (M^n\psi)^p\\
&\ge \sum_{-\infty}^\infty \tilde f(2^{k+1})^p\int_{2^k+(9/8)^n2^{k-1}}^{2^k+(9/8)^n2^k} (M^n\mathbbm{1}_{(2^k,2^{k+1}]})^p\\
&\ge \sum_{-\infty}^\infty  \tilde f(2^{k+1})^pC_p(9/8)^n2^k\\
&=C_p(9/8)^n\|\psi\|_p^p\\
&\ge \half C_p(9/8)^n\|\tilde f\|_p^p,
\eal
so by picking $n=n(p)$ sufficiently large, we get
\[
\|M^nf\|_p^p\ge \|M^n\psi\|_p^p\ge 2^{p+1}\|\tilde f\|_p^p\ge 2^p\|f\|_p^p,
\]
so
\begin{equation}\label{ineq:stronggrowth}
\|M^nf \|_p\ge 2\|f\|_p.
\end{equation}
Now suppose that $\|Mf-f\|_p<\tilde\epsilon\|f\|_p$ for some $\tilde\epsilon$ to be chosen later. From the subadditivity of the maximal operator, it follows that $\|M\phi_1-M\phi_2\|_p\le A_p\|\phi_1-\phi_2\|_p$, so
\[
\|M^nf-f\|_p\le \sum_{j=1}^n\|M^jf- M^{j-1}f\|_p\le\sum_{j=1}^nA_p^{j-1}\|Mf- f\|_p<\left(\tilde\epsilon\sum_{j=1}^nA_p^{j-1}\right)\|f\|_p
\]
which contradicts $\eqref{ineq:stronggrowth}$ for $\tilde\epsilon=\tilde\epsilon(p)$ sufficiently small. Thus $\|Mf-f\|_p\ge\tilde\epsilon\|f\|_p$, so
\[
\|Mf\|_p^p=\int (Mf)^p\ge\int f^p+(Mf-f)^p=\|f\|_p^p+\|Mf-f\|_p^p\ge\left(1+\tilde\epsilon^p\right)\|f\|_p^p,
\]
which proves the theorem.
\end{proof}

\section{Concluding Remarks}\label{sec:asymptote}
Take any compactly supported bounded function $f\geq 0$ which is not identically zero. One can show that 
\begin{align}\label{finalrr}
(9/8)^{1/p}\leq \liminf_{k \to \infty}\|M^{k}f\|^{1/k}_{L^{p}}  \leq    \limsup_{k \to \infty} \|M^{k}f\|^{1/k}_{L^{p}} \leq a_{p},
\end{align}
where
the number $a_{p}>1$ solves $M(|x|^{-1/p})=a_{p} |x|^{-1/p}$ (such an $a_p$ can be seen to exist by a calculation, or by scaling considerations). In other words, the growth of $\|M^{k}f\|_{p}$ is exponential which suggests that Theorem~\ref{cth1} is likely to be true for all $1<p<\infty$.  To show (\ref{finalrr}) let us first illustrate the upper bound. Consider the function $\tilde{f}(x):=f(Cx)/\|f\|_{\infty}$. For any fixed constant $C\neq 0$ one can easily see that 
$\limsup_{k \to \infty} \|M^{k}f\|^{1/k}_{L^{p}} = \limsup_{k \to \infty} \|M^{k}\tilde{f}\|^{1/k}_{L^{p}}$. Therefore without loss of generality we can assume that $f\leq 1$ and the support of $f$ is in $[-1,1]$. Next, take any $\delta\in (0,p-1)$, and consider 
\begin{align*}
h(x) = 
\begin{cases}
1 & |x| \leq 1,\\
|x|^{-1/(p-\delta)} & |x|>1.
\end{cases}
\end{align*}
Clearly $h \in L^{p}$, and $f\leq h$. Since $M(|x|^{-1/p})=a_{p} |x|^{-1/p}$ it follows that $M h(x) \leq a_{p-\delta} h(x)$ for all $x \in \mathbb{R}$.  Thus 
\begin{align*}
\limsup_{k \to \infty} \|M^{k} f \|_{p}^{1/k} \leq \limsup_{k \to \infty} \|M^{k} h \|_{p}^{1/k} \leq a_{p-\delta} \limsup_{k \to \infty} \|h \|_{p}^{1/k}  = a_{p-\delta}.
\end{align*}
Finally, taking $\delta \to 0$ gives the desired inequality. 

\vskip1cm

To prove the lower bound, we have already seen that the function $\bar g(x)=(1-\sqrt{x})\mathbbm{1}_{(0,1]}$ satisfies
\[
M^n\bar g(x)\ge \bar g\left((8/9)^nx\right),
\]
so we can obtain the growth $(9/8)^{n/p}$ for the function $\bar g(x)$. Now it remains to notice that for any $f \geq 0, f\in L^{p}$ not identically zero 
we can rescale and shift the function $\bar g$ so that $Mf(x)\geq A \bar g(Bx+C)$ for some constants $A>0$, $B,C\neq 0$. This finishes the proof of the claim. 

\section{Acknowledgments}
We are grateful to an anonymous referee for her/his helpful comments and suggestions.

\bibliographystyle{abbrv}
\bibliography{zbarskybib}
\end{document}